\documentclass[11pt]{amsart}
\usepackage{amsmath,amssymb,amsthm}
\usepackage{graphics,graphicx}
\usepackage{enumerate}
\usepackage{fancyhdr}
\usepackage[T1]{fontenc}
\usepackage{rotating}
\usepackage{MnSymbol}
\usepackage{tikz-cd}
\usepackage{tikz}
\usepackage[labelsep=none]{caption}
\usepackage{float}
\usepackage{chngcntr}

\counterwithout{equation}{section} 
\counterwithin{equation}{section}
\newcounter{dummy} \numberwithin{dummy}{section}

\newtheorem{theorem}[dummy]{Theorem}

\newtheorem{lemma}[dummy]{Lemma}
\newtheorem{remark}[dummy]{Remark}
\newtheorem{corollary}[dummy]{Corollary}
\newtheorem{prop}[dummy]{Proposition}

\newtheorem*{theorem*}{Theorem}

\newtheorem{wsinger}[dummy]{Conjecture}

\newtheorem*{Aconj}{Strong Atiyah Conjecture}

\theoremstyle{remark}

\newtheorem{example}{Example}[section]

\newcommand{\Q}{\mathbf{q}}
\newcommand{\T}{\mathbf{t}}

\DeclareMathOperator{\im}{im}
\DeclareMathOperator{\Lk}{Lk}

\DeclareMathOperator{\lcm}{lcm}
\DeclareMathOperator{\Link}{Link}
\DeclareMathOperator{\St}{St}
\newcommand{\vertex}{\node[vertex]}
\date{}

\title{Coxeter Groups, Ruins, and Weighted $L^2$-cohomology}

\author{Wiktor Mogilski}
\address{Department of Mathematical Sciences, Binghamton University, Binghamton, NY 13902, USA }

\email{mogilski@math.binghamton.edu}
\thanks{}

\author{Kevin Schreve}
\address{Department of Mathematics, University of Michigan, 530 Church St., Ann Arbor, MI 48109, USA}

\email{schreve@umich.edu}



\date{}

\dedicatory{}

\begin{document}

\begin{abstract}

Given a Coxeter system $(W,S)$ and a multiparameter $\Q$ of real numbers indexed by $S$, one can define the weighted $L^2$-cohomology groups and associate to them a nonnegative real number called the weighted $L^2$-Betti number. We show that for ranges of $\Q$ depending on certain subgroups of $W$, the weighted $L^2$-cohomology groups of $W$ are concentrated in low dimensions. We then prove new vanishing results for the weighted $L^2$-cohomology of certain low-dimensional Coxeter groups. Our arguments rely on computing the $L^2$-cohomology of certain complexes called ruins, as well as the resolution of the Strong Atiyah Conjecture for hyperbolic Coxeter groups. We conclude by extending to the weighted setting the computations of Davis and Okun for the case where the nerve of a right-angled Coxeter group is the barycentric subdivision of a PL-cellulation of an $(n-1)$-manifold with $n=6,8$.
\end{abstract}
\maketitle

\section{Introduction}


 $L^2$-cohomology can be defined for any CW-complex with a proper and cocompact action by an infinite discrete group. Heuristically, it comes from restricting the cellular cochains to the subspace which are square summable. These cochains naturally form a Hilbert space, and this additional structure combined with the group action allows us to put a dimension on these cohomology groups, the \emph{$L^2$-Betti numbers}.

In this paper, we will study the $L^2$-cohomology of Coxeter groups acting on the Davis complex. In fact, in this setting there is a more general \emph{weighted } cohomology theory developed in \cite{Dymara} and \cite{DDJO}. The idea behind this is natural and studied outside the Coxeter group setting: one can modify the natural norm on chains by a weight function and consider the complex of associated square summable chains. One can think of these groups as follows: when the weight function is ``small'', it is easy to be square summable, and weighted $L^2$-cohomology behaves like ordinary cohomology. On the other hand, when the weight function is ``large'', it is difficult to be square summable, and  weighted $L^2$-cohomology behaves like cohomology with compact support.

However, the Coxeter group setting is unique in that one can define  \emph{weighted $L^2$-Betti numbers}. Weighted $L^2$-Betti numbers not only tie the theory of weighted $L^2$-cohomology to algebraic properties of the Coxeter group, but they also have interesting connections with other topics such as buildings, Hecke algebras, growth series, and the Hopf Conjecture concerning Euler characteristics of aspherical manifolds.

The results of this article can be summarized as follows.

\begin{itemize}
\item Using a spectral sequence argument, we show that for ranges of weights depending on certain subgroups of the Coxeter group, the weighted $L^2$-cohomology of the Davis complex is concentrated in low dimensions (Theorem \ref{thm:l2kskeleton}). Our result generalizes a theorem of Dymara \cite[Theorem 10.3]{Dymara}.

\item We use the Strong Atiyah Conjecture and the above argument to compute (in some cases) the weighted $L^2$-cohomology for a certain subcomplex of $\Sigma(W,S)$ called a ruin. This allows us to derive vanishing results for the weighted $L^2$-cohomology of many low-dimensional Coxeter groups (Theorems \ref{thm:AtiyahKn} and \ref{thm:Atiyahcomputation}).

\item We extend (to the weighted setting) a computation of Davis-Okun \cite[Theorem 3.1]{DO3} of the $L^2$-cohomology of right-angled Coxeter groups whose nerves are barycentric subdivisions of PL $n$-manifolds for $n = 6,8$. A consequence of this is that the Weighted Singer Conjecture is true for all even $n\leq 8$ whenever the Coxeter group is right-angled and the corresponding nerve is the barycentric subdivision of a PL-cellulation of a sphere.

\end{itemize}

\emph{Acknowledgements} We would like to thank Boris Okun for many helpful discussions.

\section{Coxeter Groups and the Davis Complex}

\subsection{Coxeter groups}
A \emph{Coxeter matrix} $M=(m_{st})$ on a set $S$ is an $S\times S$ symmetric matrix with entries in $\mathbb{N}\cup\{\infty\}$ such that
$$m_{st}=\begin{cases} 1 &\mbox{if } s=t \\
\geq 2 & \mbox{otherwise.} \end{cases}$$
One can associate to $M$ a presentation for a group $W$ as follows. Let $S$ be the set of generators and let $\mathcal{I}=\{(s,t)\in S\times S \mid m_{st}\neq\infty\}.$ The set of relations for $W$ is $$R=\{(st)^{m_{st}}\}_{(s,t)\in\mathcal{I}}.$$
The group defined by the presentation $\left<S,R\right>$ is a \emph{Coxeter group} and the pair $(W,S)$ is a \emph{Coxeter system}. If all off-diagonal entries of $M$ are either $2$ or $\infty$, then $W$ is \emph{right-angled}.

Given a subset $T\subset S$, define $W_T$ to be the subgroup of $W$ generated by the elements of $T$. Then $(W_T,T)$ is a Coxeter system. Subgroups of this form are \emph{special subgroups}. $W_T$ is a \emph{spherical subgroup} if $W_T$ is finite and, in this case, $T$ is a \emph{spherical subset}. We will let $\mathcal{S}$ denote the poset of spherical subsets (the partial order being inclusion).

Given $w\in W$, call an expression $w=s_1s_2\dotsm s_n$ \emph{reduced} if there exists no integer $k<n$ with $w=s_1's_2'\dotsm s_k'.$ We define the \emph{length} of $w$, denoted by $l(w)$, to be the integer $n$ so that $w=s_1s_2\dotsm s_n$ is a reduced expression for $w$.

If $T\subset S$ and $w\in W$, the coset $wW_T$ contains a unique element of shortest length. This element $w$ is said to be \emph{$(\emptyset, T)$-reduced}.

If $(W,S)$ is a Coxeter system, let $\T := (t_s)_{s \in S}$ denote an $S$-tuple of indeterminates, where $t_s = t_{s'}$ if $s$ and $s'$ are conjugate in $W$. If $s_1s_2 \dots s_n$ is a reduced expression for $w$, let $t_w$ be the monomial $t_w := t_{s_1}t_{s_2}\dots t_{s_n}$.

Note that $t_w$ is independent of choice of reduced expression due to Tits' solution to the word problem for Coxeter groups (see the discussion at the beginning of \cite[Chapter 17]{Davis}). The \emph{growth series} of $W$ is the power series in $\mathbf{t}$ defined by $$W(\mathbf{t})=\sum_{w\in W} t_w.$$

The \emph{region of convergence} $\mathcal{R}$ for $W(\mathbf{t})$ is defined to be $$\mathcal{R}:=\{\mathbf{t}\in (0,+\infty)^S \mid W(\mathbf{t}) \text{ converges}\}.$$

For each $T\subset S$, we denote the growth series of the special subgroup $W_T$ by $W_T(\mathbf{t})$, the respective region of convergence by $\mathcal{R}_T$, and define $\mathbf{t}^{-1}:=(t_s^{-1})_{s\in S}$. 



\subsection{The basic construction} A \emph{mirror structure} over a set $S$ on a space $X$ is a family of subspaces $(X_s)_{s\in S}$ indexed by $S$. We say $X$ is a \emph{mirrored space over $S$}. Suppose that $(W,S)$ is a Coxeter system and that $X$ is a mirrored space over $S$. Put $S(x):=\{s\in S\mid x\in X_s\}$ and define an equivalence relation $\sim$ on $W\times X$ by $(w,x)\sim (w',y)$ if and only if $x=y$ and $w^{-1}w'\in W_{S(x)}$. Let $\mathcal{U}(W,X)$ denote the quotient space: $$\mathcal{U}(W,X)=(W\times X)/ \sim.$$
$\mathcal{U}(W,X)$ is the \emph{basic construction}. There is a natural $W$-action on $W\times X$ which respects the equivalence relation, and hence descends to an action on $\mathcal{U}(W,X)$.

\subsection{The Davis complex}\label{section:coxcell}

Recall that $\mathcal{S}$ the set of all spherical subsets of $S$. Let $K$ denote the geometric realization of the poset $\mathcal{S}$ and $L$ the geometric realization of the abstract simplicial complex $\mathcal{S}$. $K$ is the \textit{Davis chamber} and $L$ is called the \textit{nerve} of $(W,S)$. Note that $K$ is the cone on the barycentric subdivision of $L$ with cone point corresponding to $\emptyset$.

For each $s\in S$ let $$K_s:=|\mathcal{S}_{\geq\{s\}}|.$$ 
The family $(K_s)_{s\in S}$ is a mirror structure on $K$.

The \textit{Davis complex} $\Sigma(W,S)$ associated to the nerve $L$ is defined to be $\Sigma(W,S):=\mathcal{U}(W,K)$. Note that $\Sigma(W,S)$ is a simplicial complex. It is proved in \cite{Davis1} that $\Sigma(W,S)$ is contractible. Furthermore, if $L$ is a triangulation of an $(n-1)$-sphere, then $\Sigma(W,S)$ is an $n$-manifold.

The Davis complex admits a decomposition into \emph{Coxeter cells} as follows. For each $T\in\mathcal{S}$, let $v_T$ denote the corresponding barycenter in $K$. Let $c_T$ denote the union of simplices $c\subset\Sigma(W,S)$ such that $c\cap K_T=v_T$. The boundary of $c_T$ is cellulated by $wc_U$, where $w\in W_T$ and $U\subset T$. With its simplicial structure, the boundary $\partial c_T$ is the Coxeter complex corresponding to the Coxeter system $(W_T,T)$, which is a sphere since $W_T$ is finite. It follows that $c_T$ and its translates are cells, which are called the \emph{Coxeter cells of type $T$}. We denote $\Sigma(W,S)$ with this cellular decomposition by $\Sigma_{cc}(W,S)$. Note that $\Sigma_{cc}(W,S)$ is a regular CW-complex such that the poset of cells that can be identified with $W\mathcal{S}:=\{wW_U\mid w\in W, T\in\mathcal{S}\}$. The simplicial structure on $\Sigma(W,S)$ is the geometric realization of the poset $W\mathcal{S}$, hence $\Sigma(W,S)$ is the barycentric subdivision of $\Sigma_{cc}(W,S)$.

\section{Weighted $L^2$-cohomology}

In this section we give a brief introduction to weighted $L^2$-cohomology. Further details can be found in \cite{Davis}, \cite{DDJO}, \cite{Dymara} and \cite{eckmann}. We then compile some results about the weighted $L^2$-cohomology of the Davis complex $\Sigma(W,S)$.

Let $(W,S)$ be a Coxeter system. For the remainder of this article, let $\Q=(q_s)_{s\in S}$ denote an $S$-tuple of positive real numbers satisfying $q_s=q_{s'}$ whenever $s$ and $s'$ are conjugate in $W$. Set $\Q^{-1}=(q^{-1}_s)_{s\in S}$. If $w=s_1\cdot\cdot\cdot s_n$ is a reduced expression for $w\in W$, we define $q_w:=q_{s_1}\cdot\cdot\cdot q_{s_n}.$

\subsection{Hecke-von Neumann algebras.} Let $\mathbb{R}W$ denote the group algebra of $W$, and let $\{e_w\}_{w\in W}$ denote the standard basis on $\mathbb{R}W$ (here $e_w$ denotes the characteristic function of $\{w\}$). Given a multiparameter $\Q$ of positive real numbers as above, we deform the standard inner product on $\mathbb{R}W$ to an inner product $$\left<e_w,e_{w'}\right>_\Q=\left\{
                \begin{array}{ll}
                  q_w, & \hbox{$w=w'$;} \\
                  0, & \hbox{otherwise.}
                \end{array}
              \right.
$$

Using the multiparameter $\Q$, one can give $\mathbb{R}W$ the structure of a \textit{Hecke algebra}. We denote $\mathbb{R}W$ with this Hecke algebra structure by $\mathbb{R}_\Q (W)$, and $L^2_\Q(W)$ will denote the Hilbert space completion of $\mathbb{R}_\Q (W)$ with respect to $\left<\hskip1mm,\hskip1mm\right>_\Q$. There is a natural anti-involution on $\mathbb{R}_\Q W$, which implies that there is an associated \textit{Hecke-von Neumann algebra} $\mathcal{N}_\Q(W)$ acting on the right on $L^2_\Q(W)$. It is the algebra of all bounded linear endomorphisms of $L^2_\Q(W)$ which commute with the left $\mathbb{R}_\Q (W)$-action.

Define the $\textit{von Neumann trace}$ of $\phi\in\mathcal{N}_\Q(W)$ by $\textnormal{tr}_{\mathcal{N}_\Q}(\phi):=\left<e_1\phi,e_1\right>_\Q$. This extends to a trace on $(n\times n)$-matrices with coefficients in $\phi\in\mathcal{N}_\Q(W)$ by taking the sum of the traces of diagonal elements.

Let $V\subseteq (L^2_\Q(W))^n$ be invariant under the $W$-action. The orthogonal projection $p_V: (L^2_\Q(W))^n\rightarrow (L^2_\Q(W))^n$ is in $\mathcal{N}_\Q(W)$. Therefore, we can attribute to $V$ a nonnegative real number called the \textit{von Neumann dimension}.

\subsection{Weighted $L^2$-cohomology.} Suppose $(W,S)$ is a Coxeter system and that $X$ is a mirrored finite $CW$-complex over $S$. Set $\mathcal{U}=\mathcal{U}(W,X)$. We first orient the cells of $X$ and equivariantly extend this orientation to $\mathcal{U}$ in such a way so that if $\sigma$ is a positively oriented cell of $X$, then $w\sigma$ is positively oriented for each $w\in W$.

We define a measure on the $w$-orbit of an $i$-cell $\sigma\in X$ by $$\mu_\Q(w\sigma)=q_u,$$ where $u$ is $(\emptyset, S(\sigma))$-reduced, $S(\sigma):=\{s\in S|\sigma\subseteq X_s\}.$ This extends to a measure on the $i$-cells $\mathcal{U}^{(i)}$, which we also denote by $\mu_\Q$.

There is a weighted inner product given by $$\left<f,g\right>_\Q=\sum_\sigma f(\sigma)g(\sigma)\mu_\Q(\sigma),$$ and we denote the induced norm by $||\text{ }||_\Q$.

Define the \textit{$\Q$-weighted $i$-dimensional $L^2$-(co)chains}, denoted by $L_\Q^2C_i(\mathcal{U})$,  to be the Hilbert space of square-summable (with respect to the weighted inner product) $i$-(co)chains.

The boundary map $\partial_i: L_\Q^2C_i(\mathcal{U})\rightarrow L_\Q^2C_{i-1}(\mathcal{U})$ and coboundary map $\delta^i: L_\Q^2C_i(\mathcal{U})\rightarrow L_\Q^2C_{i+1}(\mathcal{U})$ are defined by the usual formulas, however there is one caveat: they are not adjoints with respect to this inner product whenever $\Q\neq\mathbf{1}$. One remedies this issue by perturbing the boundary map $\partial_i$ to $\partial_i^\Q$: $$\partial_i^\Q(f)(\sigma^{i-1})=\sum_{\sigma^{i-1}\subset\alpha^{i}} [\sigma:\alpha]\mu_\Q(\alpha)\mu_\Q^{-1}(\sigma) f(\alpha).$$

A simple computation shows that $\partial_i^\Q$ is the adjoint of $\delta$ with respect to the weighted inner product, hence $\left(L_\Q^2C_\ast(\mathcal{U}),\partial_i^\Q\right)$ is also a chain complex. The \textit{reduced $\Q$-weighted $L^2$-cohomology} is defined to be $$L_\Q^2H_i(\mathcal{U})=\ker \partial_i^\Q/\overline{\im \partial_{i+1}^\Q},$$ $$L_\Q^2H^i(\mathcal{U})=\ker\delta^i/\overline{\im\delta^{i-1}}.$$

\begin{lemma} We recall some properties of weighted $L^2$-cohomology \cite{DDJO}.

\begin{itemize}

\item (Hodge decomposition): $L_\Q^2H^i(\mathcal{U})= (\ker \delta^i \cap \partial_i^{\Q}) \oplus \overline{\im \delta^{i-1}} \oplus \overline{\im \partial^{\Q}_{i+1}}$

\item $L_\Q^2H^i(\mathcal{U}) \cong L_\Q^2H_i(\mathcal{U})$

\item (Poincar\'{e} Duality): $L_\Q^2H_i(\mathcal{U})\cong L_{\Q^{-1}}^2H_{n-i}(\mathcal{U},\partial\mathcal{U}).$

\item (Invariance of cellulation)\label{prop:l2coxcell} If $\Sigma_{cc}(W,S)$ denotes $\Sigma(W,S)$ with the Coxeter cellulation, then $L^2_\Q H_\ast(\Sigma(W,S))\cong L^2_\Q H_\ast(\Sigma_{cc}(W,S))$.

\end{itemize}

\end{lemma}


The \textit{$i$-th $L_\Q^2$-Betti number of $\mathcal{U}$}, denoted by $L_\Q^2b_i(\mathcal{U})$, is the von Neumann dimension of $L^2_\Q H_i(\mathcal{U})$. We define the \textit{weighted Euler characteristic of $\mathcal{U}$}: $$\chi_\Q(\mathcal{U})=\sum (-1)^iL_\Q^2b_i(\mathcal{U}).$$

\begin{lemma}\label{lemma:propertiesbetti} We record some further properties of weighted $L^2$-Betti numbers

\begin{itemize}
\item (Weighted Atiyah Formula) \cite[Corollary 3.4]{Dymara}: $\chi_\Q(\Sigma(W,S)) = \frac{1}{W(\Q)}$

\item (Weighted Kunneth Formula): $ L_\Q^2b_k(\mathcal{U} \times \mathcal{U}') = \sum_{i + j = k}L_\Q^2b_i(\mathcal{U})L_\Q^2b_k(\mathcal{U}')$

\item ($0$-dimensional homology) \cite[Theorem 10.3]{Dymara} \label{prop:betti0}
$L_\Q^2b_0(\Sigma(W,S))\neq 0$ if and only if $\Q\in\mathcal{R}$. Moreover, when $\Q\in\mathcal{R}$, $L_\Q^2b_\ast(\Sigma(W,S))$ is concentrated in dimension $0$.

\item (Top homology) \cite[Lemma 4.8]{wm2} \label{Pushing0}
Suppose that $\Sigma(W,S)$ is $n$-dimensional and that $L^2_\mathbf{1}b_n(\Sigma(W,S))=0$. Then $L^2_\Q b_n(\Sigma(W,S))=0$ for $\mathbf{q}\leq\mathbf{1}$.

\end{itemize}

\end{lemma}

\subsection{An alternate definition of $L_\Q^2$--Betti numbers}\label{sec:altbetti} As discussed in \cite[$\S$6]{DO2}, L\"{u}ck defined an algebraic version of $L_\Q^2$-Betti numbers \cite{Luckbook}. The main point is that there is an equivalence of categories between the category of Hilbert $\mathcal{N}_\Q$-modules and projective $\mathcal{N}_\Q$-modules. Hence one can define $\dim_{\mathcal{N}_\Q} M$ for a finitely generated projective $\mathcal{N}_\Q$-module $M$ which agrees with the dimension of the corresponding Hilbert $\mathcal{N}_\Q$-module. So, $\dim_{\mathcal{N}_\Q} M$ for an arbitrary $\mathcal{N}_\Q$-module is then defined to be the dimension of its projective part.

As in \cite{Luckbook}, define the cohomology groups $H_W^\ast(\mathcal{U}(W,X),\mathcal{N}_\Q(W))$ to be the cohomology of the complex $$C_W^\ast(\mathcal{U},\mathcal{N}_\Q(W)):=\text{Hom}_W(C_\ast(\mathcal{U}),\mathcal{N}_\Q(W)).$$

It follows that $$L_\Q^2b_i(\mathcal{U})=\dim_{\mathcal{N}_\Q}H_W^i(\mathcal{U},\mathcal{N}_\Q(W)).$$ The advantage of this definition is that we do not need to take closures of images (which allows us e.g. to use spectral sequences).

\subsection{Conjectures in $L^2_\Q$-cohomology.}


The following conjecture strengthens a classical conjecture of Singer, which predicts that the $L^2$-cohomology of a closed, aspherical $n$-manifold is concentrated in the middle dimension.

\begin{wsinger}[Weighted Singer Conjecture]
\label{conj:weightedSinger}
Suppose that $L$ is a triangulation of $S^{n-1}$, so that $\Sigma(W,S)$ is a contractible $n$-manifold. Then $$L_\mathbf{q}^2H_i(\Sigma(W,S))=0 \text{ for } i>\frac{n}{2} \text{ and } \mathbf{q}\leq \mathbf{1}.$$
\end{wsinger}

By weighted Poincar\'{e} duality, this is equivalent to the conjecture that if $\mathbf{q}\geq \mathbf{1}$ and $i<\frac{n}{2}$, then $L_\mathbf{q}^2H_i(\Sigma(W,S))$ vanishes. Conjecture \ref{conj:weightedSinger} is true by \cite[Theorem 2.1]{Dymara} for $n\leq 2$, and in \cite{DDJO}, it was proved for the case where $W$ is right-angled and $n\leq 4$, and furthermore it was shown that Conjecture \ref{conj:weightedSinger} for $n$ odd implies Conjecture \ref{conj:weightedSinger} for $n$ even (also under the assumption that $W$ is right-angled). Recently, the first author proved the conjecture under the assumption that $L$ is a flag triangulation and $n\leq 4$ \cite{wm}.

If $\mathbf{q}=\mathbf{1}$, these are the ordinary $L^2$-cohomology groups, and one recovers the original Singer Conjecture. The Singer Conjecture holds by a result of Lott and L\"{u}ck \cite{LottLuck}, in conjunction with the validity of the Geometrization Conjecture for $3$-manifolds \cite{p02,p03}, for $n=3$. It was proved by Davis-Okun \cite{do} for the case where $W$ is right-angled and $n\leq 4$. It was later proved for the case where $W$ is an even Coxeter group and $n=4$ by Schroeder \cite{Schroeder} under the restriction that $L$ is a flag complex. Recently Okun--Schreve \cite[Theorem 4.9]{os14} gave a proof of the Singer Conjecture for the case $n=4$, so now the Singer Conjecture is known in full generality for Coxeter groups in dimensions $n\leq 4$.

Recall that a subcomplex $A$ of $L$ is \emph{full} if the vertices of a simplex of $L$ lie in $A$, then the entire simplex lies in $A$. In \cite{DDJO}, the following key strengthening of the Weighted Singer Conjecture was proved in low dimensions:

\begin{theorem}[{\cite[Theorem 16.13]{DDJO}}]\label{RAGeneralizedSinger}
Suppose that $W_L$ is right-angled, $L$ is a triangulation of $S^{n-1}$ with $n\leq 4$ and that $A$ is any full subcomplex. Furthermore, suppose that $\mathbf{q}\leq\mathbf{1}$. Then $L^2_\Q b_i(\Sigma(W_A))=0$ for $i>\frac{n}{2}$.
\end{theorem}

\section{Ruins}\label{sec:ruins}
\subsection{Some Hilbert $\mathcal{N}_\Q(W)$-submodules of $L^2_\Q(W)$}\label{sec:hts}
We begin by considering the following self-adjoint idempotents in $\mathcal{N}_\Q(W)$:

\begin{lemma}[{\cite[Lemma 19.2.6]{Davis}}]
Given a subset $T\subset S$ and and $\Q\in\mathcal{R}^{-1}_T$, there is an idempotent $h_T\in\mathcal{N}_\Q(W)$ defined by
$$h_T:=\frac{1}{W_T(\Q^{-1})}\sum_{w\in W_T} \varepsilon_w q_w^{-1} e_w,$$
where $\varepsilon_w=(-1)^{l(w)}.$
\end{lemma}

Therefore, the map defined by $x\rightarrow h_T x$ is an orthogonal projection (whenever $h_T$ is defined), and we denote the image by $H_T$. Note that by \cite[Lemma 19.2.13]{Davis}, $$H_T=\bigcap_{s\in T} H_s.$$

Using these submodules, we define a chain complex as follows. For a spherical subset of cardinality $k$, $T\in\mathcal{S}^{(k)}$, let $$C_i(H_T):=\bigoplus_{U\in (\mathcal{S}_{\geq T})^{(i+k)}} H _U.$$

Fix some ordering of $\{s\in S-T\mid T\cup\{s\}\in\mathcal{S}\}$. Whenever $U\subset V$, we have an inclusion $i_V^U: H_V\hookrightarrow H_U$. Therefore,  the boundary map $\partial : C_{i+1}(H_T)\rightarrow C_i(H_T)$ corresponds to a matrix $(\partial_{UV})$, where $\partial_{UV}=0$ unless $U\subset V$, and is equal to $(-1)^j i_V^U$ if $U$ is obtained by deleting the $j^{\text{th}}$ element of $V$. This turns $C_\ast(H_T)$ into a chain complex of Hilbert $\mathcal{N}_\Q(W)$-modules. Similarly, whenever $U\subset V$ we have the projection $p_V^U: H_U\rightarrow H_V$. Thus we have a coboundary map where the matrix entries consist of projections, and we get a cochain complex $C^\ast(H_T)$ of Hilbert $\mathcal{N}_\Q(W)$-modules.

\subsection{Ruins} For $U\subset S$, set $\mathcal{S}(U):=\{T\in\mathcal{S}\mid T\subset U\}$. Define $\Sigma(U)$ to be the subcomplex of $\Sigma_{cc}(W,S)$ consisting of all (closed) Coxeter cells of type $T$ with $T\in\mathcal{S}(U)$. Given $T\in\mathcal{S}(U)$, we define the following subcomplexes of $\Sigma(U)$:

\begin{align*}
\Omega(U,T): & \hskip2mm \text{the union of closed cells of type }T', \text{ with }T'\in\mathcal{S}(U)_{\geq T},\\
\partial\Omega(U,T): & \hskip2mm \text{the cells of }\Omega(U,T)\text{ of type } T'', \text{ with } T''\not\in\mathcal{S}(U)_{\geq T}.
\end{align*}

The pair $(\Omega(U,T),\partial\Omega(U,T))$ is the \emph{$(U,T)$-ruin}. Note that if $T=\emptyset$, then $\Omega(U,T)=\Sigma(U)$ and $\partial\Omega(U,T)=\emptyset$. Ruins can also be expressed in terms of the basic construction. Define $K(U,T):=\Omega(U,T)\cap K$ and $\partial K(U,T):=\partial\Omega(U,T)\cap K$, where $K$ is the Davis chamber. Then $K(U,T)$ and $\partial K(U,T)$ have an induced mirror structure, and it follows that

$$\Omega(U,T)=\mathcal{U}(W,K(U,T)),\hskip2mm \partial\Omega(U,T)=\mathcal{U}(W,\partial K(U,T)).$$

The $(S,T)$-ruin has a chain complex that looks like this:

\begin{prop}[{\cite[Lemma 20.6.21]{Davis}}]
\label{prop:ruinchaincpx}
For $T\in\mathcal{S}^{(k)}$, the chain complexes $C_\ast(H_T)$ and $L_\Q^2C_{\ast+k}(\Omega(S,T),\partial\Omega(S,T))$ of $\mathcal{N}_\Q(W)$-modules are isomorphic. In particular, $$L_\Q^2C_{m}(\Omega(S,T),\partial\Omega(S,T))=0 \text{ for } m<k.$$
\end{prop}

For brevity, we write $(\Omega(U,T),\partial)$. For $s\in T$, set $U'=U-s$ and $T'=T-s$. As in \cite[Proof of Theorem 8.3]{DDJO}, we have the following weak exact sequence:
\begin{center}
\begin{tikzpicture}[baseline= (a).base]\label{eq:ruinsequence}\node[scale=1] (a) at (0,0){
\begin{tikzcd}[column sep=small]
 L_\Q^2H_{\ast}(\Omega(U',T'),\partial) \arrow{r} & L_\Q^2H_{\ast}(\Omega(U,T'),\partial) \arrow{r} & L_\Q^2H_{\ast}(\Omega(U,T),\partial)
\end{tikzcd}
};
\end{tikzpicture}
\end{center}
For the special case when $U=S$ and $T=\{s\}$ the above sequence becomes:


\begin{center}
\begin{tikzpicture}[baseline= (a).base]\node[scale=1] (a) at (0,0){
\begin{tikzcd}[column sep=small]\label{eq:ruinsequence}
  L_\Q^2H_{\ast}(\Sigma(S-s)) \arrow{r} & L_\Q^2H_{\ast}(\Sigma(S))\arrow{r} & L_\Q^2H_{\ast}(\Omega(S,s),\partial)
\end{tikzcd}
};
\end{tikzpicture}
\end{center}

\subsection{$L^2_\Q$-(co)homology of ruins}
Given a Coxeter system $(W,S)$, for $T\in\mathcal{S}$ and $T\subseteq V\subseteq S$, define $$\St(T,V):=\bigcup_{\substack{U\subseteq V \\ U\cup T\in\mathcal{S}}} U,$$ and $$Lk(T,V):=\St(T,V)\setminus T.$$ If $V=S$, then we write $\St(T)$ and $Lk(T)$ instead of $\St(T,S)$ and $Lk(T,S)$. If $T=\emptyset$, we make the convention that $S(T,U)=U$.

\begin{theorem}[Compare {\cite[Theorem 20.6.22]{Davis}}]
\label{thm:homologykruin}
Suppose that $T\in\mathcal{S}^{(k)}$ and that $\Q\in \mathcal{R}_{\St(T)}$. Then $L^2_\Q H_\ast(\Omega(S,T),\partial\Omega(S,T))$ is concentrated in dimension $k$.

\end{theorem}
\begin{proof}
 We first make an observation about ruins. We note that for every $V\subseteq S$, $\Omega(V,T)=\Omega(\St(T,V),T)$, the point being that $\Omega(V,T)$ consists of Coxeter cells corresponding to spherical subsets of $V$ containing $T$, and $\St(T,V)$ is the union of all such subsets. In particular, $\Omega(U,T)=\Omega(\St(T,U),T))$, and hence $\Omega(U,T)$ is a subcomplex of $\Sigma(\St(T,U))$.

The proof is by induction on $k$. We will show that for $U\subset S$ and $T\in\mathcal{S}(U)^{(k)}$, $L^2_\Q H_\ast(\Omega(U,T),\partial)$ is concentrated in dimension $k$. For the base case $k=0$, note that $\St(\emptyset,U)=U$, $\Omega(U,\emptyset)=\Sigma(U)$, and $\partial\Omega(U,\emptyset)=\emptyset$. Hence, for $k=0$, the theorem asserts that for $\Q\in\mathcal{R}_U$, $L^2_\Q H_\ast(\Sigma(U))$ is concentrated in dimension $0$, which is Lemma \ref{prop:betti0}.

Suppose the theorem is true for $k-1$ and let $T\in\mathcal{S}(U)^{(k)}$. Let $s\in T$, $V=T-s$ and consider the long exact sequence:


$$\longrightarrow L_\Q^2H_\ast(\Omega(\St(T,U)-s,V),\partial)\longrightarrow L_\Q^2H_\ast(\Omega(\St(T,U),V),\partial) \longrightarrow\dotsm$$
$$
\dotsm\longrightarrow L_\Q^2H_\ast(\Omega(\St(T,U),T),\partial)\longrightarrow
$$

Note that $$\Omega(\St(T,U),V)=\Omega(\St(V,\St(T,U)),V)$$ and $$\Omega(\St(T,U)-s,V)=\Omega(\St(V,\St(T,U)-s),V).$$

Since $\St(V,\St(T,U))\subseteq \St(T,U)$ and $\St(V,\St(T,U)-s)\subseteq \St(T,U)$, it follows that $\mathcal{R}_{\St(T,U)}\subseteq\mathcal{R}_{\St(V,\St(T,U))}$ and $\mathcal{R}_{\St(T,U)}\subseteq\mathcal{R}_{\St(V,\St(T,U)-s)}$. Since $\Q\in \mathcal{R}_{\St(T,U)}$, it follows by induction that the left-hand term and the middle term of the exact sequence are both concentrated in dimension $k-1$. Since $L_\Q^2H_\ast(\Omega(\St(T,U),T),\partial)$ vanishes for $\ast<k$ (Proposition \ref{prop:ruinchaincpx}), it follows that $L_\Q^2H_\ast(\Omega(\St(T,U),T),\partial)=L_\Q^2H_\ast(\Omega(U,T),\partial)$ is concentrated in dimension $k$.

\end{proof}

\begin{remark}\label{remark:concentrationsubruins}
Suppose that $T\in\mathcal{S}$ and that $\Q\in \mathcal{R}_{\St(T)}$. Then, for $U\in \mathcal{S}_{\geq T}$, Theorem \ref{thm:homologykruin} implies that $L^2_\Q H_\ast(\Omega(S,U),\partial)$ is concentrated in dimension $|U|$. This is because, if $T\subset U$, then $\St(U)\subset St(T)$. Therefore $\Q\in \mathcal{R}_{\St(T)}\subseteq\mathcal{R}_{\St(U)}$.

\end{remark}





\section{A spectral sequence} In this section, we define a spectral sequence following the line laid down in \cite[$\S$ 2]{DO2}.

A \textit{poset of coefficients} is a contravariant functor $\mathcal{A}$ from a poset $\mathcal{P}$ to the category of abelian groups. In other words, it is a collection $\{\mathcal{A}\}_{a\in\mathcal{P}}$ of abelian groups together with homomorphisms $\phi_{ba}: \mathcal{A}_a\rightarrow\mathcal{A}_b$, defined whenever $a>b$, such that $\phi_{ca}=\phi_{cb}\phi_{ba}$, whenever $a>b>c$. The functor $\mathcal{A}$ gives us a system of coefficients on the cell complex Flag($\mathcal{P}$): it associates to the simplex $\sigma$ the abelian group $\mathcal{A}_{\min(\sigma)}$. Hence, we get a cochain complex $$C^j(\text{Flag}(\mathcal{P});\mathcal{A}):=\bigoplus_{\sigma\in \text{Flag}(\mathcal{P})^{(j)}}\mathcal{A}_{\min(\sigma)},$$

where $\text{Flag}(\mathcal{P})^{(j)}$ denotes the set of $j$-simplices in Flag($\mathcal{P}$).

Let $Y$ be a CW complex. A \textit{poset of spaces} in $Y$ over $\mathcal{P}$ is a cover $\mathcal{V}=\{Y_a\}_{a\in\mathcal{P}}$ of $Y$ by subcomplexes indexed by $\mathcal{P}$ so that if $N(\mathcal{V})$ denotes the nerve of the cover, then:

\begin{enumerate}[(i)]
  \item $a<b \Longrightarrow Y_a\subset Y_b$,
  \item the vertex set Vert($\sigma$) of each simplex in $N(\mathcal{V})$ has the greatest lower bound $\wedge\sigma$ in $\mathcal{P}$, and
  \item $\mathcal{V}$ is closed under taking finite nonempty intersections, i.e., for any simplex $\sigma$ of $N(\mathcal{V})$, $$\bigcap_{a\in\sigma} Y_a=Y_{\wedge\sigma}.$$
\end{enumerate}

Note that any cover leads to a poset of spaces by taking all nonempty intersections as elements of the new cover and removing duplicates. The resulting poset is the set of all nonempty intersections, ordered by inclusion.

The following lemmas appearing in \cite{DO2} define a spectral sequence associated to a poset of spaces, and give conditions for the sequence to degenerate.

\begin{lemma}[{\cite[Lemma 2.1]{DO2}}]
\label{lemma:spectralseq}
Suppose $\mathcal{V}=\{Y_a\}_{a\in\mathcal{P}}$ is a poset of spaces for $Y$ over $\mathcal{P}$. There is a Mayer-Vietoris type spectral sequence converging to $H^*(Y)$ with $E_1$-term: $$E_1^{i,j}=C^i(\text{Flag}(\mathcal{P}); \mathcal{H}^j(\mathcal{V})),$$ and $E_2$-term: $$E_2^{i,j}=H^i(\text{Flag}(\mathcal{P}); \mathcal{H}^j(\mathcal{V})),$$ where the coefficient system $\mathcal{H}^j(\mathcal{V})$ is given by $\mathcal{H}^j(\mathcal{V})(\sigma)=H^j(Y_{\min(\sigma)})$.
\end{lemma}

\begin{lemma}[{\cite[Lemma 2.2]{DO2}}]
\label{lemma:spectralseqdeg}
Suppose that $\mathcal{V}:=\{Y_a\}_{a\in\mathcal{P}}$ is a poset of spaces for $Y$ over $\mathcal{P}$. If for every $a\in\mathcal{P}$, the induced homomorphism $H^\ast(Y_a)\rightarrow H^\ast(Y_{<a})$ is the zero map, then the spectral sequence degenerates at $E_2$ and $$H^\ast(Y)=\bigoplus_{a\in\mathcal{P}}H^i(\text{Flag}(\mathcal{P}_{\geq a}),\text{Flag}(\mathcal{P}_{>a}),H^j(Y_a)).$$
\end{lemma}

\section{$L_\Q^2$-(co)homology of $(\Sigma,\Sigma^{(k-1)})$}

To simplify notation, write $\Sigma$ for $\Sigma_{cc}(W,S)$ and let $\Sigma^{(k-1)}$ denote the $(k-1)$-skeleton of $\Sigma_{cc}(W,S)$. For the proofs in this section, we will also write $H^\ast_\Q(\mathcal{U})$ for $H_W^\ast(\mathcal{U},\mathcal{N}_\Q(W))$. Refer to Section \ref{sec:altbetti} for the definition of $H_W^\ast(\mathcal{U},\mathcal{N}_\Q(W))$.

The following lemma is standard.
\begin{lemma}
\label{lemma:homologyskeltonsimplex}
Let $\Delta$ denote the standard $n$-simplex and let $\Delta^{(k)}$ be its $k$-skeleton, $k<n$. Then the reduced homology $\tilde{H}_\ast(\Delta^{(k)})$ (with coefficients in $\mathbb{R}$) is concentrated in dimension $k$. 
\end{lemma}


\begin{theorem}
\label{thm:l2kskeleton}
Let $k\geq 1$. Suppose that for every $T\in\mathcal{S}^{(k)}$, $\Q\in\mathcal{R}_{St(T)}$, and let $\Sigma^{(k-1)}$ denote the $(k-1)$-skeleton of $\Sigma$. Then $L_\Q^2 H_\ast(\Sigma,\Sigma^{(k-1)})$ is concentrated in dimension $k$. 

\end{theorem}
\begin{proof}
We will show that $H_\Q^\ast(\Sigma,\Sigma^{(k-1)})$ is concentrated in dimension $k$.

Consider the relative cochain complex $L^2_\Q C^\ast(\Sigma,\Sigma^{(k-1)})$. We have that
$$L^2_\Q C^i(\Sigma,\Sigma^{(k-1)})=\left\{                                                      \begin{array}{ll}
                                                                    0, & \hbox{$i\leq k-1$;} \\
                                                                    \bigoplus_{U\in\mathcal{S}^{(i)}} H_U, & \hbox{$i>k-1$.}
                                                                  \end{array}
                                                                \right.$$
Set $C_{-i}(\Sigma,\Sigma^{(k-1)})=L^2_\Q C^i(\Sigma,\Sigma^{(k-1)})$. Then $C_\ast(\Sigma,\Sigma^{(k-1)})$ is a chain complex. For every $T\in\mathcal{S}$ let $\tilde{\Omega}_T$ denote the complex with reindexed chain complex $C_{-i}(\tilde{\Omega}_T)=L^2_\Q C^i(\Omega(S,T),\partial\Omega(S,T))$. In this way, we have made the cochain complex of every $(S,T)$-ruin a subcomplex of the re-indexed relative cochain complex $C_\ast(\Sigma,\Sigma^{(k-1)})$.

Let $\mathcal{P}$ be the poset $\mathcal{S}^{(\geq k)}$ with the order reversed. By the above re-indexing of cochain complexes, $\{\tilde{\Omega}_T\}_{T\in\mathcal{P}}$ is a poset of spaces over the complex $Y$ with chain complex $C_\ast(\Sigma,\Sigma^{(k-1)})$, and hence we have the spectral sequence of Lemma \ref{lemma:spectralseq}.

We first establish the condition of Lemma \ref{lemma:spectralseqdeg}. So, we claim that for every $U\in\mathcal{P}$, the induced map $H_\Q^\ast(\tilde{\Omega}_U)\rightarrow H_\Q^\ast(\tilde{\Omega}_{<U})$ is the zero map. To prove the claim we will show that $H_\Q^{-|U|}(\tilde{\Omega}_{<U})=0$, as this implies $H_\Q^\ast(\tilde{\Omega}_U)\rightarrow H_\Q^\ast(\tilde{\Omega}_{<U})$ is the zero map since $H_\Q^\ast(\tilde{\Omega}_U)$ is concentrated in dimension $-|U|$ (see Remark \ref{remark:concentrationsubruins}). Recall that $\tilde{\Omega}_{<U}=\bigcup_{T\in\mathcal{P}_{<U}}\tilde{\Omega}_{T}$. The proof is by induction. For the base case, note that for every spherical $V$ properly containing $U$, the induced map $H_\Q^\ast(\tilde{\Omega}_U)\rightarrow H_\Q^\ast(\tilde{\Omega}_{V})$ is the zero map. This is because of Theorem \ref{thm:homologykruin}, which states that $H_\Q^\ast(\tilde{\Omega}_U)$ and $H_\Q^\ast(\tilde{\Omega}_{V})$ are concentrated in dimension $-|U|$ and $-|V|$, respectively, and $-|V|<-|U|$. Let $\mathcal{C}$ be a subcollection of elements of $\mathcal{P}_{<U}$ and let $B=\bigcup_{T\in\mathcal{C}}\tilde{\Omega}_{T}$. We wish to show that $H_\Q^{-|U|}(B)=0$. Write $B=A\cup \tilde{\Omega}_V$, where $V\in\mathcal{C}$ and $A=\bigcup_{\substack{T\in\mathcal{C}\\ T\neq V}}\tilde{\Omega}_T$. Then we have the Mayer-Vietoris sequence:

$$\begin{tikzcd}[column sep=small]
 H_\Q^{-|U|-1}(A\cap \tilde{\Omega}_V) \arrow{r} &  H_\Q^{-|U|}(B) \arrow{r} & H_\Q^{-|U|}(A)\oplus H_\Q^{-|U|}(\tilde{\Omega}_V)
\end{tikzcd}$$

By induction, $H_\Q^\ast(A)$ vanishes for $\ast\geq -|U|$, and by Theorem \ref{thm:homologykruin}, $H_\Q^\ast(\tilde{\Omega}_V)$ is concentrated in $-|V|<-|U|$. We now claim $H_\Q^\ast(A\cap \tilde{\Omega}_V)$ vanishes for $\ast\geq -|U|-1$, as this implies $H_\Q^{-|U|}(B)=0$. We observe that

$$A\cap\tilde{\Omega}_V=\bigcup_{\substack{T\in\mathcal{C}\\ T\neq V}} \tilde{\Omega}_{T}\cap\tilde{\Omega}_V=\bigcup_{\substack{T\in\mathcal{C}\\ T\cup V\in\mathcal{S}}} \tilde{\Omega}_{T\cup V}.$$

The last equality follows from the fact that $H_T\cap H_V=H_{T\cup V}$ whenever $T\cup V$ is spherical (see Section \ref{sec:hts}). Thus $A\cap \tilde{\Omega}_V$ is the union of elements of the form $\tilde{\Omega}_{T\cup V}$, where $|T\cup V|\geq |V|+1> |U|+1$. By Theorem \ref{thm:homologykruin}, $H_\Q^\ast(\tilde{\Omega}_{T\cup V})$ vanishes for $\ast\geq -|U|-1$, so by induction $H_\Q^{-|U|-1}(A\cap \tilde{\Omega}_V)=0$.

We have established the condition in Lemma \ref{lemma:spectralseqdeg}, and hence
\begin{equation}\label{eq:totalhom}H_\Q^{-n}(Y)=\bigoplus_{U\in\mathcal{P}}H^{|U|-n}(\text{Flag}(\mathcal{P}_{\geq U}),\text{Flag}(\mathcal{P}_{>U}); H_\Q^{-|U|}(\tilde{\Omega}_U))\end{equation}

The strategy of the proof is as follows. By Theorem \ref{thm:homologykruin}, for every $U\in\mathcal{P}$, $H_\Q^\ast(\tilde{\Omega}_U)$ is concentrated in dimension $-|U|$. So, we are done if we show that for every $U\in\mathcal{P}$, $H^\ast(\text{Flag}(\mathcal{P}_{\geq U}),\text{Flag}(\mathcal{P}_{>U}))$ is concentrated in dimension $|U|-k$. This implies $E^{i,j}_2=0$ unless $i+j=-k$, and by (\ref{eq:totalhom}), $H_\Q^\ast(Y)$ is concentrated in dimension $-k$. Re-indexing our complexes, it follows that the cohomology of the complex $L^2_\Q C^\ast(\Sigma,\Sigma^{(k-1)})$ is concentrated in dimension $k$.

We claim that for $U\in\mathcal{P}$ with $m=|U|$, $H^\ast(\text{Flag}(\mathcal{P}_{\geq U}),\text{Flag}(\mathcal{P}_{>U}))$ is concentrated in dimension $m-k$.

Since the geometric realization of $\text{Flag}(\mathcal{P}_{\geq U})$ is contractible, by the long exact sequence for the pair it suffices to show that the reduced cohomology $\tilde{H}^\ast(\text{Flag}(\mathcal{P}_{>U}))$ is concentrated in dimension $m-k-1$. Note that if $m=k$, then we are done since $\text{Flag}(\mathcal{P}_{> U})=\emptyset$. Also, note that for the special case where $m-k=1$, the map $H^0(\text{Flag}(\mathcal{P}_{\geq U}))\rightarrow H^0(\text{Flag}(\mathcal{P}_{>U}))$ in the long exact sequence for the pair $(\text{Flag}(\mathcal{P}_{\geq U}),\text{Flag}(\mathcal{P}_{> U}))$ is injective ($\text{Flag}(\mathcal{P}_{\geq U})$ is the cone on $\text{Flag}(\mathcal{P}_{>U})$), so showing $\tilde{H}^\ast(\text{Flag}(\mathcal{P}_{>U}))$ is concentrated in dimension $m-k-1=0$ does in fact suffice.

Consider the poset $\mathcal{S}^{op}$, which is the poset $\mathcal{S}$ of spherical subsets with order reversed. Note that $\text{Flag}(\mathcal{S}^{op}_{> U})\cong \text{Flag}(\mathcal{S}_{< U})$. The geometric realization of $\text{Flag}(\mathcal{S}_{< U})$ is $b\Delta$, where $b\Delta$ is the barycentric subdivision of the $(m-1)$-dimensional simplex $\Delta$. This is because $\mathcal{S}_{< U}$ is the poset of proper subsets of $U$. Note that $\text{Flag}(\mathcal{P}_{>U})$ is a subcomplex of $\text{Flag}(\mathcal{S}^{op}_{> U})$, and more precisely, the geometric realization of $\text{Flag}(\mathcal{P}_{> U})$ is the subcomplex of barycentric subdivision of $\partial\Delta$ (recall that $k\geq 1$) obtained by removing barycenters corresponding to spherical subsets of cardinality less than $k$.

\begin{figure}[H]
\setlength{\tabcolsep}{15pt}
\begin{tabular}{c c c}
$$\begin{tikzpicture}[scale=1]
\tikzstyle{vertex}=[circle, fill=black, draw, inner sep=0pt, minimum size=2pt]
\tikzstyle{every node}=[circle,inner sep=0pt,minimum size=0pt]
\vertex (x) at (0,0) {};
\vertex (y) at (2,0) {};
\vertex (z) at (1,1.5) {};
\vertex (c) at (1, .5) {};

\path
    (x) edge node[pos=.5, circle, fill=black, draw, inner sep=0pt, minimum size=2pt] (xy) {} (y)
    (y) edge node[pos=.5, circle, fill=black, draw, inner sep=0pt, minimum size=2pt] (yz) {} (z)
    (z) edge node[pos=.5, circle, fill=black, draw, inner sep=0pt, minimum size=2pt] (xz) {} (x)
    (x) edge (c)
    (y) edge (c)
    (z) edge node[pos=1,below=2pt]{$\emptyset$} (c)
    (xy) edge (c)
    (yz) edge (c)
    (xz) edge (c)
    ;

    \node [below=.05cm] at (0,0)
        {
            $\{s\}$
        };
    \node [below=.02cm] at (1,0)
        {
            $\{s,t\}$
        };
    \node [below=.05cm] at (2,0)
        {
            $\{t\}$
        };
    \node [above=.05cm] at (1,1.5)
        {
            $\{u\}$
        };
    \node [right=.05cm] at (yz)
        {
            $\{t,u\}$
        };
    \node [left=.05cm] at (xz)
        {
            $\{s,u\}$
        };

\end{tikzpicture}$$ & $$\begin{tikzpicture}[scale=1]
\tikzstyle{vertex}=[circle, fill=black, draw, inner sep=0pt, minimum size=2pt]
\tikzstyle{every node}=[circle,inner sep=0pt,minimum size=0pt]
\vertex (x) at (0,0) {};
\vertex (y) at (2,0) {};
\vertex (z) at (1,1.5) {};

\path
    (x) edge node[pos=.5, circle, fill=black, draw, inner sep=0pt, minimum size=2pt] (xy) {} (y)
    (y) edge node[pos=.5, circle, fill=black, draw, inner sep=0pt, minimum size=2pt] (yz) {} (z)
    (z) edge node[pos=.5, circle, fill=black, draw, inner sep=0pt, minimum size=2pt] (xz) {} (x)

    ;

    \node [below=.05cm] at (0,0)
        {
            $\{s\}$
        };
    \node [below=.02cm] at (1,0)
        {
            $\{s,t\}$
        };
    \node [below=.05cm] at (2,0)
        {
            $\{t\}$
        };
    \node [above=.05cm] at (1,1.5)
        {
            $\{u\}$
        };
    \node [right=.05cm] at (yz)
        {
            $\{t,u\}$
        };
    \node [left=.05cm] at (xz)
        {
            $\{s,u\}$
        };

\end{tikzpicture}$$ & $$\begin{tikzpicture}[scale=1]
\tikzstyle{vertex}=[circle, fill=black, draw, inner sep=0pt, minimum size=2pt]
\tikzstyle{every node}=[circle,inner sep=0pt,minimum size=0pt]
\node (x) at (0,0) {};
\node (y) at (2,0) {};
\node (z) at (1,1.5) {};

\path[transparent]
    (x) edge node[pos=.5, circle, fill=black, draw, inner sep=0pt, minimum size=1pt] (xy) {} (y)
    (y) edge node[pos=.5, circle, fill=black, draw, inner sep=0pt, minimum size=1pt] (yz) {} (z)
    (z) edge node[pos=.5, circle, fill=black, draw, inner sep=0pt, minimum size=1pt] (xz) {} (x)

    ;

    \node [below=.02cm] at (1,0)
        {
            $\{s,t\}$
        };
    \node [right=.05cm] at (yz)
        {
            $\{t,u\}$
        };
    \node [left=.05cm] at (xz)
        {
            $\{s,u\}$
        };
    \vertex at (xy) {};
    \vertex at (xz) {};
    \vertex at (yz) {};

\end{tikzpicture}$$\\
$\text{Flag}(\mathcal{S}^{op}_{> U})$ & $\text{Flag}(\mathcal{P}_{> U})$, $k=1 $ &$\text{Flag}(\mathcal{P}_{> U})$, $k=2 $
\end{tabular}
\caption{:\hskip2mm Geometric realizations when $U=\{s,t,u\}$}
\end{figure}
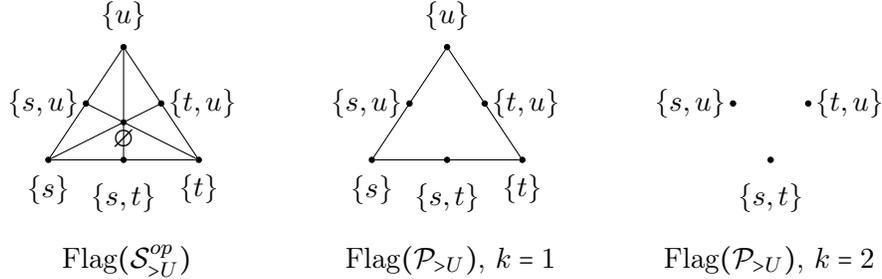

These barycenters correspond to faces of $\partial\Delta$ of dimension less than or equal to $k-2$. Hence $$\tilde{H}^\ast(\text{Flag}(\mathcal{P}_{> U}))\cong \tilde{H}^\ast(\partial\Delta-\Delta^{(k-2)}),$$ where $\Delta^{(k-2)}$ denotes the $(k-2)$-skeleton of $\Delta$. Note that if $k=1$, then $\Delta^{(k-2)}=\emptyset$, and we are done as the reduced homology $\tilde{H}^\ast(\partial\Delta)$ is concentrated in dimension $m-2$. So, suppose $k>1$. By Alexander Duality, $$\tilde{H}_\ast(\Delta^{(k-2)})\cong \tilde{H}^{m-\ast-3}(\partial\Delta-\Delta^{(k-2)}).$$ By Lemma \ref{lemma:homologyskeltonsimplex}, the reduced homology $\tilde{H}_\ast(\Delta^{(k-2)})$ is concentrated in dimension $k-2$. 
 It follows that $\tilde{H}^\ast(\partial\Delta-\Delta^{(k-2)})$ is concentrated in dimension $m-k-1$, and therefore the same holds for $\tilde{H}^\ast(\text{Flag}(\mathcal{P}_{>U}))$.




\end{proof}

\begin{remark}\label{remark:ruinconcentration}
In fact, the proof of Theorem \ref{thm:l2kskeleton} implies the following statement: if for every $U\in\mathcal{S}^{\geq k}$ we have that $L^2_\Q H_\ast(\Omega(S,U),\partial)$ is concentrated in dimension $|U|$, then $L_\Q^2 H_\ast(\Sigma,\Sigma^{(k-1)})$ is concentrated in dimension $k$. Thus, if one can show the aforementioned concentration of ruins, then Corollary \ref{cor:vanishingstars} implies that $L_\Q^2 H_n(\Sigma)=0$ for $n>k$.
\end{remark}

\begin{corollary}
\label{cor:vanishingstars}
Suppose that for every $T\in\mathcal{S}^{(k)}$, $\Q\in\mathcal{R}_{St(T)}$. Then $$L_\Q^2 H_n(\Sigma)=0 \text{ for } n>k.$$

\end{corollary}
\begin{proof}
For the case where $k=0$, this is just Proposition \ref{prop:betti0}, so suppose $k\geq 1$. By Theorem \ref{thm:l2kskeleton}, $L_\Q^2 H_\ast(\Sigma,\Sigma^{(k-1)})$ is concentrated in dimension $k$. Therefore the long exact sequence for the pair $(\Sigma,\Sigma^{(k-1)})$ implies the assertion.
\end{proof}

We shall need the following corollary in Section $7$.

\begin{corollary}\label{cor:ruins2dim}
Suppose that $\Sigma(W,S)$ is two-dimensional and that for every $t\in S, L_\mathbf{1}^2H_\ast(\Omega(S,t),\partial)$ is concentrated in dimension $1$. Then $$L_\mathbf{q}^2 H_2 (\Sigma(W,S))=0 \text{ for } \Q\leq\mathbf{1}.$$
\end{corollary}
\begin{proof} Let $T\in\mathcal{S}^{(2)}$. As $\Sigma_L$ is two-dimensional, Proposition \ref{prop:ruinchaincpx} implies that $L_\mathbf{1}^2H_\ast(\Omega(S,T),\partial)$ is concentrated in dimension $2$. This and the assumption imply that all ruins are concentrated in the correct dimension, and hence $L_\mathbf{1}^2 H_2 (\Sigma(W,S))=0$ (see Remark \ref{remark:ruinconcentration}). The statement for $\Q\leq\mathbf{1}$ follows from Lemma \ref{lemma:propertiesbetti}.
\end{proof}






\begin{example}

Suppose that $L$ is a flag triangulation of a closed orientable surface, and let $(W,S)$ be the associated right-angled Coxeter group. Following \cite[Example 17.4.3]{Davis}, we have that

$$\frac{1}{W(t)}=1-(f_0-3)t+(f_0+3-3\chi(L))t^2-(\chi(L)-1)t^3,$$

where $\chi(L)$ is the Euler characteristic of $L$ and $f_0$ is the number of vertices of $L$. The radius of convergence $\rho$ of $W(t)$ is the smallest modulus of a root of the above polynomial. 

The link of every vertex of $L$ is a $k$-gon, with $k\geq 4$. If $W_{Lk(v)}$ denotes the special subgroup corresponding the the link of a vertex $v$, and $Lk(v)$ has $k$ vertices, then by an easy computation we have that

$$\rho_{Lk(v)}=\frac{(k-2)-\sqrt{k^2-4k}}{2}.$$

If $v_0$ is the vertex of $L$ whose link has the maximal number of vertices, then Corollary \ref{cor:vanishingstars} implies that $L_q^2 H_\ast(\Sigma(W,S))=0$ is concentrated in dimension $1$ whenever $\rho<q<\rho_{Lk(v_0)}$. 

The main point is that $\rho$ is explicitly computable. For example, it is still open if $L_q^2 b_\ast(\Sigma(W,S))$ is concentrated in degree $2$ for $q=1$ (see \cite[Conjecture 11.5.1]{do}), but Corollary \ref{cor:vanishingstars} lets us conclude that $L_q^2 b_2(\Sigma(W,S))=0$ for $q<\rho_{Lk(v_0)}$.

\end{example}

\section{Further computations}

\subsection{The Strong Atiyah Conjecture and weighted $L^2$-Betti numbers} In this section, we will assume that the nerve $L$ of the Coxeter system $(W,S)$ is a graph (so that $\Sigma(W,S)$ is two-dimensional). Furthermore, we need only consider the Davis complex with the Coxeter cellulation. For convenience, we will label the edge $[s_is_j]$ of the nerve with $m_{ij}$. This allows us to recover the Coxeter system by reading the presentation off of the nerve.

\begin{Aconj} Let $G$ be a discrete group. For any proper cocompact $G$-CW complex $Y$
	\[
		\lcm G \cdot L^2_\mathbf{1}b_k(Y;G) \in \mathbb{Z},
	\]
	where $\lcm G$ denotes the least common multiple of the orders of finite subgroups of $G$.
\end{Aconj}

In \cite{ks}, the second author proved the Strong Atiyah Conjecture for groups that are virtually cocompact special. Restricting to the Coxeter group setting, we have the following:

\begin{theorem}[{\cite[Corollary 6.5]{ks}}]\label{thm:atiyah} Let $W$ be a Coxeter group without Euclidean triangle special subgroups. If $W$ is a finite index subgroup of $W'$, then the Strong Atiyah Conjecture holds for $W'$.
\end{theorem}

We will exploit Theorem \ref{thm:atiyah} to make computations of $L^2$-Betti numbers for many examples of Coxeter groups. For these examples, the strategy will be to show that the $L^2$-cohomology of every $(S,t)$-ruin is concentrated in dimension $1$. We can then apply Corollary \ref{cor:ruins2dim} to compute the $L^2$-cohomology of the Davis complex.

For every $t\in S$, the $(S,t)$-ruin has the property that $\Omega(S,t)=\Omega(St(t),t)$. Therefore, we can compute the $L^2$-Betti numbers $L_\mathbf{1}^2b_\ast(\Omega(S,t),\partial)$ using the subgroup $W_{St(t)}$. 

Suppose the labeled $\St(t)$ has an automorphism group $G$ that is vertex transitive and edge transitive. Furthermore, suppose $G$ does not flip edges, so that we can think of the edge stabilizer as a subgroup of the vertex stabilizer. 

Since $G$ acts by labeled automorphism on $\St(t)$, it acts on $W_{St(t)}$. We can form the semidirect product $W_{St(t)} \rtimes G$. Obviously, the semidirect product also acts on the Davis complex $\Sigma(W_{\St(t)})$. Furthermore, the following lemma of Davis tells us that the stabilizers of Coxeter cells under this action are only the obvious subgroups.

\begin{lemma}[{\cite[Proposition 9.1.9]{Davis}}]\label{prop:stabilizers}
For $U\in\mathcal{S}(St(t))$ and $w\in W_{St(t)}$, let $wc_U$ denote the Coxeter cell of type $U$, and let $G_U$ denote the stabilizer of the simplex $U$ in $\St(t)$ under the $G$-action. The isotropy subgroup of $W_{St(t)} \rtimes G$ fixing $wc_U$ is the subgroup $wW_Uw^{-1} \rtimes G_U$.
\end{lemma}

Let $G_t$ denote the stabilizer of the vertex $t$ in $\St(t)$ under the $G$-action. Then the group $W_{St(t)} \rtimes G_t$ acts on the $(St(t),t)$-ruin. By multiplicativity of $L^2$-Betti numbers we have the following.

\begin{lemma}\label{lemma:multiplicative}
 $$L_\mathbf{1}^2b_\ast(\Omega(St(t),t),\partial, W_{St(t)} \rtimes G_t) = |G_t|L^2_\mathbf{1}b_\ast(\Omega(St(t),t),\partial,W_{St(t)}).$$
\end{lemma}

The following simple lemma will allow us to use the Strong Atiyah Conjecture to make computations.

\begin{lemma}\label{lcm12}
Suppose that $X$ is a $G$-complex with associated $L^2$-chain complex $$0 \rightarrow L^2C_n(X) \rightarrow L^2C_{n-1}(X) \rightarrow \dots \rightarrow L^2C_1(X) \rightarrow L^2C_0(X) \rightarrow 0.$$ If $\partial_i: L^2C_i(X) \rightarrow L^2C_{i-1}(X)$ or $\partial_i: L^2C_i(X) \rightarrow L^2C_{i-1}(X)$ is nonzero, and $\dim_G(L^2C_i(X)) = \frac{1}{\lcm(G)}$, and $G$ satisfies the Strong Atiyah Conjecture, then $L^2H_i(X,G) = 0$.
\end{lemma}

\begin{proof}
Since $G$ satisfies the Strong Atiyah Conjecture, we know that if $L^2b_i(X,G) < \frac{1}{\lcm(G)}$, then $L^2b_i(X,G) = 0$. Since one of the boundary maps is nonzero, $L^2H_i(X,G)$ is strictly smaller than $L^2C_i(X)$.

By assumption, $\dim_G(L^2C_i(X)) = \frac{1}{\lcm(G)}$, and this implies $L^2H_i(X,G) = 0$.
\end{proof}

\begin{lemma}\label{vanishing}
Suppose that $\St(t)$ is labeled with $m$'s and let $\deg(t)$ denote the degree of the vertex $t$ in $\St(t)$. Furthermore, suppose that $G_t$ acts transitively on the edges emanating from $t$. If $m$ is a multiple of $\deg(t)$, then $L_\mathbf{1}^2H_\ast(\Omega(S,t),\partial)$ is concentrated in dimension $1$.
\end{lemma}
\begin{proof}
By Proposition \ref{prop:ruinchaincpx}, we have that $$\dim_{W_{St(t)}} L_\mathbf{1}^2 C_2(\Omega(S,t),\partial,W_{St(t)})=\frac{\deg(t)}{2m}.$$

If $\frac{\deg(t)}{2m} = \frac{|G_t|}{\lcm(W_{St(t)} \rtimes G_t)}$, then $L_\mathbf{1}^2b_2(\Omega(St(t),t),\partial, W_{St(t)} \rtimes G_t)=0$ by Lemma \ref{lcm12}, and hence Lemma \ref{lemma:multiplicative} implies $$L_\mathbf{1}^2b_2(\Omega(S,t),\partial)=L^2_\mathbf{1}b_2(\Omega(St(t),t),\partial,W_{St(t)})=0.$$

Let $E_T$ denote the stabilizer of the submodule $H_T$ of the chain complex $L_\mathbf{1}^2 C_2(\Omega(St(t),t),\partial,W_{St(t)})$ under the $W_{St(t)} \rtimes G_t$-action. Since $G_t$ acts transitively on the edges adjacent to $t$, it acts transitively on the submodules $H_T$, and hence $|G_t|/|E_{T}|=\deg(t)$. If $m$ is a multiple of $\deg(t)$, then $$\lcm (2|G_t|,2m|E_T|)=2|G_t|\frac{m}{\deg(t)}.$$

It follows that $\lcm(W_{St(t)} \rtimes G_t)=2|G_t|\frac{m}{\deg(t)}$ and this completes the proof.

\end{proof}

\begin{theorem}\label{thm:AtiyahKn}
Let $K_n$ denote the complete graph on $n$-vertices. Suppose that the nerve $L=K_n$ is labeled by uniform labels $m\geq n-1$. Then $$L_\mathbf{q}^2 H_2 (\Sigma(W,S))=0 \text{ for } \Q\leq\mathbf{1}.$$
\end{theorem}
\begin{proof}
 It suffices to consider the case where $m = n-1$; if $m > n-1$ then the group $W_{K_n}$ is a Coxeter subgroup of $W_{K_{m+1}}$, and vanishing for the larger group implies vanishing of the subgroup in the top dimension. For every $t\in S$, the alternating group $A_{n-1}$ acts on $St(t)$ and the action is transitive on the edges adjacent to $t$. Therefore, by Lemma \ref{vanishing}, we obtain that the $L^2$-cohomology of every $(S,t)$-ruin is concentrated in dimension $1$. The theorem now follows by Corollary \ref{cor:ruins2dim}.



\end{proof}

\begin{theorem}\label{thm:Atiyahcomputation}
Suppose the labeled nerve $L$ with uniform labels $m$ is a graph with no $3$-cycles. If $m$ is a multiple of $\lcm_{s\in S}\deg(s)$, then $$L_\mathbf{q}^2 H_2 (\Sigma(W,S))=0 \text{ for } \Q\leq\mathbf{1}.$$
\end{theorem}
\begin{proof}
If $L$ has no $3$-cycles, then for every $t\in S$, $Lk(t)$ is a full subcomplex of $L$. Since $Lk(t)$ consists of $\deg(t)$ points, it follows that the alternating group $A_{\deg(t)}$ acts on the $(St(t),t)$-ruin, and moreover, the action is transitive on the edges of $St(t)$. Again, by Lemma \ref{vanishing}, we obtain that every $(S,t)$-ruin has $L^2$-cohomology concentrated in dimension $1$, and the theorem follows by Corollary \ref{cor:ruins2dim}.

\end{proof}

\subsection{Weighted $L^2$-homology of barycentric subdivisions}
In this section, all Coxeter groups are assumed to be right-angled. Unless otherwise stated, the standing assumption in this section is that $\mathbf{q}\leq\mathbf{1}$. For simplicity, we will use the following notation: if $A$ is a full subcomplex of the nerve $L$, then $$b_i^\mathbf{q}(L):=L^2_\Q b_i(\Sigma(W,S))\text{ and } b_i^\mathbf{q}(A):=L^2_\Q b_i(W_L\Sigma_A).$$

Let $L$ be a regular CW-complex and let $bL$ denote its barycentric subdivision. Since $bL$ is a flag complex, there is a natural right-angled Coxeter group associated to $bL$ denoted by $W_{bL}$. The vertices of $bL$ are naturally graded by ``dimension'': each vertex $v$ is the barycenter of a unique cell $\sigma_v$. The dimension of $v$ is the dimension of $\sigma_v$. For a $k$-cell $\sigma^k$, we denote its boundary complex by $\partial\sigma^k$. We begin with the following set of lemmas.

\begin{lemma}[{\cite[Lemma 2.1]{DO3}}]\label{Links}
Let $bL$ be the barycentric subdivision of a regular CW-complex $L$, and let $v$ be a vertex of dimension $k$. Then $$\Link_{bL}(v)=b\partial\sigma_v^k\ast b\Link_L \sigma_v^k.$$
\end{lemma}

\begin{remark}\label{LinksRemark}
Suppose that $L$ is a PL-cellulation of an $(n-1)$-manifold, then $\Link_{bL}(v)=S^{k-1}\ast S^{n-2-k},$ where $S^{k-1}$ and $S^{n-2-k}$ are barycentric subdivisions of PL-cellulations of spheres of corresponding dimensions.
\end{remark}

\begin{lemma}[{\cite[Lemma 3.2]{DO3}}]\label{AffectingLinks}
Let $bL$ be the barycentric subdivision of a regular CW-complex $L$, and let $v$ be a vertex of dimension $k$. The removal of $v$ from $bL$ does not change the first factor in the decomposition of Lemma \ref{Links} of links for vertices of dimension less or equal than $k$, and it does not change the second factor for vertices of dimension greater or equal than $k$.
\end{lemma}

\begin{lemma}\label{HomologyAffectedLink}
Suppose that $A$ is a full subcomplex of a triangulation of $S^m$, where $m\leq 3$. If Conjecture \ref{conj:weightedSinger} holds for $S^p$, then $b_i^\mathbf{q}(S^p\ast A)=0$ for $i>\frac{p+m}{2}+1$.
\end{lemma}
\begin{proof}
By Theorem \ref{RAGeneralizedSinger}, $b_i^\mathbf{q}(A)=0$ for all $i>\frac{m+1}{2}$. Since Conjecture \ref{conj:weightedSinger} holds for $S^p$, it follows that $b_i^\mathbf{q}(S^p)=0$ for $i>\frac{p+1}{2}$. The assertion follows by the weighted K\"{u}nneth formula.
\end{proof}

\begin{lemma}\label{RemovingLinks}
Suppose the nerve $L$ is $(n-1)$-dimensional and that $v$ is a $k$-vertex of $bL$. Let $bL-v$ denote the complex obtained by removing the vertex $v$ from $bL$. If $\Link_{bL}(v)$ satisfies the hypothesis of Lemma \ref{HomologyAffectedLink}, then $$b_i^\mathbf{q}(bL)=b_i^\mathbf{q}(bL-v)\text{ for } i>\frac{n}{2}.$$
\end{lemma}
\begin{proof}
Set $\Lk(v)=\Link_{bL}(v)$ and let $C\Lk(v)$ denote the right-angled cone on $\Lk(v)$. Consider the Mayer-Vietoris sequence:
 $$\begin{tikzcd}[column sep=small]
  L^2_\mathbf{q}H_\ast(\Sigma_{\Lk(v)}) \arrow{r} & L^2_\mathbf{q}H_\ast(\Sigma_{C\Lk(v)}) \oplus L^2_\mathbf{q}H_\ast(\Sigma_{bL-v}) \arrow{r} & L^2_\mathbf{q}H_\ast(\Sigma_{bL})
\end{tikzcd}$$

By assumption, $L^2_\mathbf{q}H_\ast(\Sigma_{\Lk(v)})=0$ for $\ast>\frac{n-1}{2}$, and hence by the weighted K\"{u}nneth formula, $L^2_\mathbf{q}H_\ast(\Sigma_{C\Lk(v)})=0$ for $\ast>\frac{n-1}{2}$. The lemma follows from the weak exactness of the Mayer-Vietoris sequence.
\end{proof}

In \cite{DO3}, Davis and Okun computed the ordinary $L^2$-Betti numbers of the right-angled Coxeter group $W_{bL}$ where $L$ is a PL-cellulation of an $(n-1)$-manifold for $n = 6,8$: they show that they always vanish above the middle dimension. We prove the analogous statements for weighted $L^2$-cohomology.

\begin{theorem}\label{WeightedManifold}
Suppose that $L$ is a PL-cellulation of an $(n-1)$-manifold with $n=6,8$. Then $b_i^\mathbf{q}(bL)=0$ for $i>\frac{n}{2}$.
\end{theorem}

\begin{proof}
We first consider the case where $n=6$. By Lemma \ref{Links}, vertices of dimension $3$ and $1$ have respective links $S^2\ast S^{1}$ and $S^0\ast S^{3}$, and the link of a vertex of dimension $4$ is $S^3\ast S^{0}$ (all spheres involved are barycentric subdivisions of PL-cellulations of spheres). By the validity of Conjecture \ref{conj:weightedSinger} for $n\leq 4$, these links satisfy the hypothesis of Lemma \ref{HomologyAffectedLink}. We obtain a $2$-complex $T$ by removing vertices of $L$ of dimensions $4$, $3$, $1$ (in that order). Note that Lemma \ref{AffectingLinks} implies that all of the first factors of the join decompositions of the links are not affected. By Lemma \ref{RemovingLinks}, $b_i^\mathbf{q}(bL)=b_i^\mathbf{q}(T)$ for $i>3$. Since $T$ is $2$-dimensional, the theorem follows.

For the case $n=8$, we invoke a similar procedure to obtain a $4$-complex $T$. We first begin by removing vertices of dimension $6$ (with respective links in $bL$ equal to $S^5\ast S^{0}$), then vertices dimension $4$ (with respective links in $bL$ equal to $S^3\ast S^{2}$), and finally of dimension $3$ (with respective links in $bL$ equal to $S^2\ast S^{3}$). By the previous computation for the case $n=6$ and the validity of Conjecture \ref{conj:weightedSinger} for $n\leq 4$, we have that, at each stage of removing a vertex, the remaining affected links satisfy Lemma \ref{HomologyAffectedLink}. Hence, Lemma \ref{RemovingLinks} implies that $b_i^\mathbf{q}(bL)=b_i^\mathbf{q}(T)$ for $i>4$.

 \cite[Theorem 3.1]{DO3} implies that $b_5^\mathbf{1}(T)=b_5^\mathbf{1}(bL)=0$. Since $\Sigma_{T}$ is $5$-dimensional, Lemma \ref{Pushing0} implies that $b_5^\mathbf{q}(T)=0$ for $\mathbf{q}\leq\mathbf{1}$, thus completing the proof of the assertion.
\end{proof}

\begin{remark}
In fact, the statement of Theorem \ref{WeightedManifold} holds when $L$ is assumed to be a PL-cell complex that is a generalized homology manifold. The point is that, given a $k$-vertex, its link  $\Link_{bL}(v)=S^{k-1}\ast GHS^{n-2-k},$ where $S^{k-1}$ is a $(k-1)$-sphere and $GHS^{n-2-k}$ is a generalized homology $(n-2-k)$-sphere. Moreover, Theorem \ref{RAGeneralizedSinger} and Conjecture \ref{conj:weightedSinger} (for $n\leq 4$) remain true when $L$ is assumed to be a generalized homology sphere. \end{remark}

In fact, if we assume $L$ to be a PL-cellulation of a sphere, we get the following corollary.

\begin{corollary}
Suppose that $L$ is a PL-cellulation of $S^{n-1}$. Then Conjecture \ref{conj:weightedSinger} is true for $W_{bL}$ for all even $n\leq 8$.
\end{corollary}

\end{document}